\documentclass{amsart}

\usepackage{graphicx}

\newtheorem{theorem}{Theorem}[section]

\theoremstyle{definition}

\newtheorem{corollary}{Corollary}

\theoremstyle{remark}
\newtheorem{remark}[theorem]{Remark}

\numberwithin{equation}{section}

\begin{document}

\title{On the number of monic admissible polynomials in the ring $\mathbb{Z}[x]$}

\author{T. Agama}
\address{Department of Mathematics, African Institute for Mathematical science, Ghana
}
%\curraddr{Department of Mathematics and Statistics,
%{Case Western Reserve University, Cleveland, Ohio 43403}
\email{theophilus@aims.edu.gh/emperordagama@yahoo.com}

%\author{Author Two}
%\address{Mathematical Research Section, School of Mathematical Sciences,
%Australian National University, Canberra ACT 2601, Australia}

%    General info

\date{\today}

\dedicatory{}


\begin{abstract}
 In this paper we study admissible polynomials. We establish an estimate for the number of admissible polynomials of degree $n$ with coeffients $a_i$ satisfying $0\leq a_i\leq H$ for a fixed $H$, for $i=0,1,2, \ldots, n-1$. In particular, letting $\mathcal{N}(H)$ denotes the number of monic admissible polynomials of degree $n\geq 3$ with coefficients satisfying the inequality $0\leq a_i\leq H$, we show that \begin{align}\frac{H^{n-1}}{(n-1)!}+O(H^{n-2})\leq \mathcal{N}(H) \leq \frac{n^{n-1}H^{n-1}}{(n-1)!}+O(H^{n-2}).\nonumber
\end{align} Also letting $\mathcal{A}(H)$ denotes the number of monic irreducible admissible polynomials, with coefficients satisfying the same condition , we show that \begin{align}\mathcal{A}(H)\geq \frac{H^{n-1}}{(n-1)!}+O\bigg( H^{n-4/3}(\log H)^{2/3}\bigg).\nonumber 
\end{align}
\end{abstract}

\maketitle

\section{Introduction and problem statement}
Let us consider the polynomial \begin{align}f(x)=a_nx^n+a_{n-1}x^{n-1}+\cdots +a_1x+a_0\nonumber
\end{align}of degree $n$ in the ring $\mathbb{R}[x]$. Then $f(x)$ is said to be admissible if \begin{align}n!=\sum \limits_{i=0}^{n}a_i=a_0+a_1+\cdots +a_{n-1}+a_n. \nonumber
\end{align} Let $a_n=1$ and let $\mathcal{N}(H)$ denotes the number of admissible monic polynomials belonging to the ring $\mathbb{Z}[x]$. Interest is on the number of such monic irreducible polynomial of a given degree under certain constraint. Admissible polynomials, by their nature, form an important class of polynomials. In some sense admissible polynomials gives us much information about the distribution of the coefficients. These polynomials becomes very useful in practice, because it allows us to recover with some precision the possible coefficients of any such polynomials. Through out this paper, using a sieve theoritical technique, we will be concerned with the model problem of estimating the number of monic irreducible admissible polynomials that can be formed from any given constraint on the coefficients. Letting $\mathcal{A}(H)$ denotes the number of irreducible admissible polynomials \begin{align}a_0+a_1x+\cdots +a_{n-1}x^{n-1}+x^n\nonumber
\end{align} with $0\leq a_i\leq H$ for $i=0,1,\ldots,n-1$ and $a_i\in \mathbb{Z}[x]$, we ask the question of how small can this quantity be? This paper will be concerned with addressing such a problem. But before then, we seek to find the counting function for the number admissible monic polynomials in $\mathbb{Z}[x]$. We obtain a lower bound in the following sequel.

\section{Notations}
Through out this paper a prime number will either be denoted by $p$ or $q$. Any other letter will be clarified. The quantity $\mathcal{A}_p:=\{a_n:a_n\equiv 0\pmod p\}$ for $\mathcal{A}=(a_n)$, and $S(\mathcal{A},\rho, z):=\# (\mathcal{A}\setminus \cup_{p|P(z)}\mathcal{A}_p)$, where $\rho$ is the set of all primes. The inequality $|k(n)|\leq Mp(n)$ for sufficiently large values of $n$ will be compactly written as $k(n)\ll p(n)$ or $k(n)=O(p(n))$. Similarly the inequality $|k(n)|\geq Mp(n)$ for sufficiently large values of $n$ will be represented by $k(n)\gg p(n)$. The limit $\lim \limits_{n\longrightarrow \infty}\frac{k(n)}{p(n)}=0$ will be represented in a compact form as $k(n)=o(p(n))$ as $n\longrightarrow \infty$. Also by $k(n)\asymp p(n)$, we mean there exist some constant $c_1,c_2>0$ such that $c_1p(n)\leq k(n)\leq c_2p(n)$. The quantity $\delta$ or any of it subscripts are positive numbers that are taken to be  small.

\section{Preliminary results}
\begin{theorem}(Chebychev)\label{Chebychev}
Let $\pi(z):=\sum \limits_{p\leq z}1$, then there exist some constants $c_1,c_2>0$ such that  \begin{align}c_1\frac{z}{\log z}\leq \pi(z)\leq c_2\frac{z}{\log z}. \nonumber
\end{align}
\end{theorem}

\begin{proof}
For a proof, see for instance \cite{tenenbaum2015introduction}.
\end{proof}

\begin{remark}
Now we state a very classical theorem concerning the distribution of irreducible monic polynomials in the ring $\mathbb{F}_p[x]$, which will play a crucial role in our subsequent works. It comes in the following sequel.
\end{remark}

\begin{theorem}\label{irreducible}
Let $N_{n}$ denotes the number of monic irreducible polynomials of degree $n$ in $\mathbb{F}_{p}[x]$. Then \begin{align}N_n=\frac{p^n}{n}+O(p^{n/2}).\nonumber
\end{align}
\end{theorem}

\begin{proof}
For a proof, See for instance \cite{cojocaru2005introduction}.
\end{proof}

\begin{remark}
Next we state a sifting technology due to Tur$\acute{a}$n, which will play a crucial role in obtaining an estimate for the number monic irreducible polynomials with coefficient that can be controlled.
\end{remark}

\begin{theorem}(Tur$\acute{a}$n)\label{Turan}
Let us set \begin{align}U(z):=\sum \limits_{p|P(z)}\delta_{p},\nonumber
\end{align}where $0\leq \delta_{p}<1$. Then \begin{align}S(\mathcal{A}, \mathcal{\rho},z)\leq \frac{| \mathcal{A}|}{U(z)}+\frac{2}{U(z)}\sum \limits_{p|P(z)}|R_{p}|+\frac{1}{U^2(z)}\sum \limits_{\substack{p,q|P(z)}}|R_{p,q}|,\nonumber 
\end{align}where \begin{align}P(z)=\prod \limits_{\substack{p<z\\p\in \rho}}p, \quad | \mathcal{A}_{p}|=\delta_{p}|\mathcal{A}|+R_{p}.\nonumber 
\end{align}
\end{theorem}

\begin{proof}
For a proof, See for instance \cite{cojocaru2005introduction}.
\end{proof}

\section{Main results}
\begin{theorem}\label{admissible 1}
Let $\mathcal{N}(H)$ denotes the number of polynomials $x^n+a_{n-1}x^{n-1}+\cdots +a_0$ in $\mathbb{Z}[x]$, satisfying $a_0+a_1+\cdots +a_{n-1}=n!-1$ and $0\leq a_i\leq H$ for $i=0,1,\ldots n-1$. Then \begin{align}\frac{H^{n-1}}{(n-1)!}+O(H^{n-2})\leq \mathcal{N}(H) \leq \frac{n^{n-1}H^{n-1}}{(n-1)!}+O(H^{n-2}). \nonumber
\end{align}In particular, there exist some constant $\frac{1}{(n-1)!}<c<\frac{n^{n-1}}{(n-1)!}$, such that \begin{align}\mathcal{N}(H)=(1+o(1))cH^{n-1},\nonumber 
\end{align}as $H\longrightarrow \infty$.
\end{theorem}

\begin{proof}
Consider the polynomial $x^n+a_{n-1}x^{n-1}+\cdots +a_0$, with coefficients satisfying the conditions $0\leq a_i\leq H$ and\begin{align}1+a_{n-1}+a_{n-2}+\cdots +a_0=n!.\nonumber
\end{align}We let each of this polynomials corresponds to elements of the set \begin{align}\mathcal{M}=\{(a_{0},a_1,\ldots, a_{n-1})|a_0+a_1+\cdots +a_{n-1}=n!-1,~0\leq a_{i}\leq 	H\}.\nonumber
\end{align}We remark that $\mathcal{N}(H)$ is the number of elements of the set $\mathcal{M}$. To obtain these bounds for the counting function $\mathcal{N}(H)$, we first observe that $H\leq n!$. For suppose $H>n!$, then we find that $H>n!>n!-1=a_0+a_1+a_2+\cdots a_{n-1}$. This contradicts the inequality $a_0+a_1+a_2+\cdots a_{n-1}\leq Hn$, since $n\geq 3$. For if these two inequalities hold, then we would have $Hn\leq H$, and it follows that $n\leq 1$ contradicting the inequality $n\geq 3$.  Thus the inequality \begin{align}\lfloor H\rfloor -1\leq n!-1=a_0+a_1+\ldots +a_{n-1}\leq \lfloor Hn\rfloor +1\nonumber
\end{align}is valid. The lower bound is obtained by finding the number of possible representations of the form \begin{align}a_0+a_1+\cdots +a_{n-1}=\lfloor H\rfloor-1:=K. \nonumber
\end{align}Letting $R_n(K)$ denotes the number of such different representations, then we claim that \begin{align}R_n(K)=\binom{K-1}{n-1}.\nonumber
\end{align}To see this, consider the power series\begin{align}l(z)&=\sum \limits_{K=0}^{\infty}z^{K} \nonumber
\end{align}valid in the unit disc $|z|<1$. Then it follows that \begin{align}l^{n}(z)=\sum \limits_{K=0}^{\infty}R_{n}(K)z^{K}.\nonumber
\end{align}On the other hand, we observe that \begin{align}l^n(z)&=\frac{1}{(n-1)!}\frac{d^{n-1}}{dz^{n-1}}\bigg(\frac{1}{1-z}\bigg)\nonumber \\&=\frac{1}{(n-1)!}\frac{d^{n-1}}{dz^{n-1}}\bigg(\sum \limits_{K=0}^{\infty}z^{K}\bigg)\nonumber \\&=\sum \limits_{K=n-1}^{\infty}\frac{K(K-1)\cdots (K-n+2)}{(n-1)!}z^{K-n+1}\nonumber \\&=\sum \limits_{K=n-1}^{\infty}\binom{K}{n-1}z^{K-n+1}\nonumber \\&=\sum \limits_{K=0}^{\infty}\binom{K+n-1}{n-1}z^{K}.\nonumber
\end{align}By comparism and using the fact that $R_n(K)=R_n(K-n)$, the claimed lower bound follows immediately. The upper bound follows by finding the number of different representations of the form \begin{align}a_0+a_1+\cdots +a_{n-1}=\lfloor Hn\rfloor +1,\nonumber
\end{align}by adapting the same argument, and the proof of the theorem is complete. 
\end{proof}

\begin{remark}
The above result does gives us an order of growth of monic admissible polynomials with carefully controlled coefficients. The next result highlights this very fact.
\end{remark}

\begin{corollary}\label{admissible 2}
Let $\mathcal{N}(H)$ denotes the number of monic admissible polynomials of degree $n$ in $\mathbb{Z}[x]$, with coefficients satisfying $0\leq a_{i}\leq H$ for $i=0, 1,\ldots, n-1$, then \begin{align}\mathcal{N}(H)\asymp H^{n-1}. \nonumber
\end{align}
\end{corollary}

\begin{proof}
The result follows from Theorem \ref{admissible 1}.
\end{proof}
\bigskip
We recall that there are $H^{n}$ monic polynomials of degree $n$ with coefficients satisfying $0\leq a_{i}\leq H$. Corollary \ref{admissible 2} also indicates that the number of admissible monic polynomials of degree $n$ with carefully controlled coefficients as before is of the order $H^{n-1}$. Thus when a polynomial is chosen  at random, with coefficients controlled by the quantity $H$, the probability that it is admissible or the proportion that is admissible must be roughly $\frac{c}{H}$, where $c=c(n)$. We state the next result, which gives us a lower bound for the number monic irreducible admissible polynomials, with carefully controlled coefficients.
\bigskip

\begin{theorem}
Let $\mathcal{A}(H)$ denotes the number of monic admissible irreducible polynomials of degree $n\geq 3$ in the ring $\mathbb{Z}[x]$, with coefficients satisfying the relation $0\leq a_i\leq H$ for $i=0, 1,\ldots n-1$ and a fixed $H$. Then \begin{align}\mathcal{A}(H)\geq \frac{H^{n-1}}{(n-1)!}+O\bigg( H^{n-4/3}(\log H)^{2/3}\bigg)\nonumber 
\end{align}for $H\leq n!$, and where the implied constant depends on $n$.
\end{theorem}

%% The correct journal style for \specialsection is all uppercase; a known bug
%% in amsart.cls prevents this, so input must be uppercase until it is fixed.
%\specialsection*{This is a Special Section Head}
%\specialsection*{THIS IS A SPECIAL SECTION HEAD}
%This is an example of a special section head%
%%%%%%%%%%%%%%%%%%%%%%%%%%%%%%%%%%%%%%%%%%%%%%%%%%%%%%%%%%%%%%%%%%%%%%%%
\footnote{
\par
.}%
%%%%%%%%%%%%%%%%%%%%%%%%%%%%%%%%%%%%%%%%%%%%%%%%%%%%%%%%%%%%%%%%%%%%%%%%
.
\begin{proof}
We adopt the traditional technique by way of estimating first the number of monic admissible irreducible polynomials modulo a prime $p$. By letting $\mathcal{N}(H)$ denotes the number of monic admissible polynomials in $\mathbb{Z}[x]$, we find  by Theorem \ref{admissible 1} that \begin{align}\mathcal{N}(H)&\geq \binom{\lfloor H\rfloor-2}{n-1}\nonumber \\&=\frac{H^{n-1}}{(n-1)!}+O\bigg(H^{n-2}\bigg).
\end{align}In particular, by Theorem \ref{admissible 1}, we find that $\mathcal{N}(H)=cH^{n-1}+O(H^{n-2})$, for some $c:=c(n)>0$. Let $z=z(H)$ be a real number to be chosen later and consider the polynomial \begin{align}g(x):=x^{n}+a_{n-1}x^{n-1}+\cdots +a_0 \in \mathbb{Z}[x]
\end{align}satisfying the condition $1+a_{n-1}+\cdots +a_0=n!$, for each $0\leq a_i\leq H$. We let each of these polynomials corresponds to an element of the set\begin{align}\mathcal{R}=\left\{(a_{n-1},a_{n-2},\ldots, a_{0})|a_{n-1}+a_{n-2}+\cdots a_{0}=n!-1, \quad 0\leq a_{i}\leq H\right\}.\nonumber
\end{align} Let $\mathcal{R}_{p}$ be a set of monic polynomials whose elements corresponds to polynomials in $\mathcal{R}$ whose elements are irreducible modulo $p$. Now we remark that if a polynomial is irreducible in $\mathcal{R}_{p}$ for some prime $p$, then it is irreducible in $\mathcal{R}$. We observe that the number of polynomials in $\mathcal{R}$ that corresponds to each polynomial $g(x)\pmod p$ in $\mathcal{R}_p$ is given by \begin{align}\frac{c(1+o(1))}{H}\bigg(\frac{H}{p}+O(1)\bigg)^{n},\nonumber
\end{align}where $c=c(n)$. Letting $z^2< H$, we can write \begin{align}\frac{c(1+o(1))}{H}\bigg(\frac{H}{p}+O(1)\bigg)^{n}&=\frac{(1+o(1))cH^{n-1}}{p^n}+O\bigg(\frac{H^{n-2}}{p^{n-1}}\bigg).\nonumber
\end{align}Appealing to Theorem \ref{irreducible}, we find that the number of polynomials in $\mathcal{R}$ that correspond to polynomials in $\mathcal{R}_p$ is given by \begin{align}|\mathcal{R}_{p}|&=\bigg(\frac{(1+o(1))cH^{n-1}}{p^n}+O\bigg(\frac{H^{n-2}}{p^{n-1}}\bigg)\bigg)\bigg(\frac{p^n}{n}+O(p^{n/2})\bigg)\nonumber \\&=\frac{(1+o(1))cH^{n-1}}{n}+O\bigg(\frac{H^{n-1}}{p^{n/2}}\bigg)+O(H^{n-2}p). \nonumber
\end{align}We see in relation to Theorem \ref{Turan}, that  \begin{align}\delta_{p}=\frac{1}{n},\quad R_p=\frac{H^{n-1}}{p^{n/2}}+H^{n-2}p,\nonumber 
\end{align} and \begin{align}R_{p,q}=\frac{H^{n-1}}{p^{n/2}}+\frac{H^{n-1}}{q^{n/2}}+H^{n-2}pq,
\end{align}so that by appealing to Theorem \ref{Chebychev}, we find that \begin{align}U(z)=\sum \limits_{p|P(z)}\delta_p=\sum \limits_{p|P(z)}\frac{1}{n}\gg \frac{z}{\log z}.\nonumber 
\end{align}We find that \begin{align}\frac{|\mathcal{R}|}{U(z)}&\ll \frac{H^{n-1}\log z}{z}.
\end{align}Again \begin{align}\frac{2}{U(z)}\sum \limits_{p|P(z)}|R_{p}|&=\frac{2}{U(z)}\sum \limits_{p|P(z)}\bigg(\frac{H^{n-1}}{p^{n/2}}+H^{n-2}p\bigg)\nonumber \\& \ll \frac{H^{n-1}\log z}{z}\sum \limits_{p|P(z)}\frac{1}{p^{n/2}}+\frac{H^{n-2}\log z}{z}\sum \limits_{p|P(z)}p\nonumber \\&\ll \frac{H^{n-1}\log z}{z}+H^{n-2}z\nonumber.\label{1}
\end{align}Similarly we find that \begin{align}\frac{1}{U^2(z)}\sum \limits_{\substack{p|P(z)\\q|P(z)}}|R_{p,q}|&=\frac{\log ^2 z}{z^2}\sum \limits_{\substack{p|P(z)\\q|P(z)}}\bigg(\frac{H^{n-1}}{p^{n/2}}+\frac{H^{n-1}}{q^{n/2}}+H^{n-2}pq\bigg)\nonumber \\&=\frac{H^{n-1}\log^2 z}{z^2}\sum \limits_{\substack{p|P(z)\\q|P(z)}}\frac{1}{p^{n/2}}+\frac{H^{n-1}\log^2 z}{z^2}\sum \limits_{\substack{p|P(z)\\q|P(z)}}\frac{1}{q^{n/2}}\nonumber \\&+\frac{H^{n-2}\log^2 z}{z^2}\sum \limits_{\substack{p|P(z)\\q|P(z)}}pq. \nonumber
\end{align}Thus, we find that \begin{align}\frac{1}{U^2(z)}\sum \limits_{\substack{p|P(z)\\q|P(z)}}R_{p,q}&\ll \frac{H^{n-1}\log z}{z}+H^{n-2}z^2,\nonumber
\end{align}where the implied constant depends on $n$. It follows, by Theorem \ref{Turan} that \begin{align}S(\mathcal{R}, \rho, z)&\ll \frac{H^{n-1}\log z}{z}+H^{n-2}z^2.\nonumber
\end{align}By choosing $z:=H^{1/3}(\log H)^{1/3}$, it follows that \begin{align}S(\mathcal{R},\rho, z)&=O\bigg(H^{n-4/3}(\log H)^{2/3}\bigg),\nonumber
\end{align}and the result follows immediately.
\end{proof}

\section{Final remarks}
In this paper we have been able to quantify at the very least the number of admissible polynomials in the ring $\mathbb{Z}[x]$; in particular, we have shown that the number of admissible polynomials with coefficients controlled by the quantity $H$ is of the order $\asymp H^{n-1}$, and that the number of monic reducible admissible polynomials is of the order\begin{align}\ll H^{n-4/3}(\log H)^{2/3}. \nonumber
\end{align}Aside making an improvement to the following quantitative bounds, admissible polynomials has a property that could be usefull in other areas of research, most especially in the area of cryptography.

%\begin{table}[ht]
%\caption{}\label{eqtable}
%\renewcommand\arraystretch{1.5}
%\noindent\[
%\begin{array}{|c|c|c|}
%\hline
%&{-\infty}&{+\infty}\\
%\hline
%{f_+(x,k)}&e^{\sqrt{-1}kx}+s_{12}(k)e^{-\sqrt{-1}kx}&s_{11}(k)e^
%{\sqrt{-1}kx}\\
%\hline
%{f_-(x,k)}&s_{22}(k)e^{-\sqrt{-1}kx}&e^{-\sqrt{-1}kx}+s_{21}(k)e^{\sqrt
%{-1}kx}\\
%\hline
%\end{array}
%\]
%\end{table}

%\begin{figure}[tb]
%\blankbox{.6\columnwidth}{5pc}
%\includegraphics{lion.png}
%\caption{This is an example of a figure caption with text.}
%\label{firstfig}
%\end{figure}

%\begin{figure}[tb]
%\blankbox{.75\columnwidth}{3pc}
%\includegraphics{lion.png}
%\caption{}\label{otherfig}
%\end{figure}

\bibliographystyle{amsplain}

\end{document}